\documentclass[12pt,twoside,reqno]{amsart}

\usepackage[OT1]{fontenc}    %
\usepackage{type1cm}         
\usepackage[english]{babel}  %
\usepackage{amsthm}

\numberwithin{equation}{section}
\numberwithin{figure}{section}

\theoremstyle{plain}
\newtheorem{theorem}{Theorem}[section]

\newtheorem{lemma}[theorem]{Lemma}

\theoremstyle{remark}
\newtheorem*{remark}{Remark}

\newtheorem*{ack}{Acknowledgement}

\theoremstyle{definition}

\newcommand{\R}{\mathbb{R}}

\newcommand{\N}{\mathbb{N}}

\newcommand{\HH}{\mathcal{H}}

\newcommand{\iii}{\mathtt{i}}

\newcommand{\dimh}{\textrm{dim}_H}

\newcommand{\pallo}{\circ}
\newcommand{\pois}{\backslash}

\newcommand{\eps}{\varepsilon}

\newcommand{\khii}{\text{\lower -.4ex\hbox{$\chi$}}}
\renewcommand{\emptyset}{\textrm{\O}}

\newcommand{\fii}{\varphi}
\newcommand{\Fii}{\Phi}

\begin{document}

\title[On the geometric structure of the limit set of CIFS]{On
the geometric structure of the limit set of conformal iterated
function systems}

\author{Antti K\"aenm\"aki}
\address{Department of Mathematics and Statistics \\
         P.O. Box 35 (MaD) \\
         FIN-40014 University of Jyv\"askyl\"a \\
         Finland}
\email{antakae@maths.jyu.fi\newpage\quad \thispagestyle{empty}}
\newpage\quad \thispagestyle{empty}

\thanks{Author is supported by the Academy of Finland, project 23795, and
during his stay in Universitat Aut\`onoma of Barcelona he was
supported by the Marie Curie Fellowship, grant
HPMT-CT-2000-00053. He also wishes to thank the UAB, where this paper
was completed, for warm hospitality.}
\date{20th December, 2001.}

\begin{abstract}
We consider infinite conformal iterated function systems on $\R^d$. We
study the geometric structure of the limit set of such systems.
Suppose this limit set intersects some $l$--dimensional
$C^1$--submanifold with positive Hausdorff $t$--dimensional
measure, where $0<l<d$ and $t$ is the Hausdorff dimension of the
limit set. We then show that the closure of the limit set
belongs to some $l$--dimensional affine subspace or geometric sphere
whenever $d$ exceeds $2$ and analytic curve if $d$ equals $2$.
\end{abstract}

\maketitle

\section{Introduction}

We work on the setting introduced by Mauldin and Urba\'nski in
\cite{mu}. There they consider uniformly contractive countable
collections of conformal injections defined on some open, bounded
and connected set $\Omega \subset \R^d$. It is needed that there
exists some compact set $X \subset \Omega$ with non--empty
interior such that each contraction maps this set to some subset
of itself. For this kind of setting, even without the conformality
assumption, there is always so called limit set associated. We
denote it with $E$ and we are particularly interested in the
properties of this set. The conformality assumption is basically
needed for good behavior of the derivatives. As usual, open set
condition (OSC), introduced by Moran in \cite{mo}, is used for
getting decent separation for those previously mentioned subsets
of $X$. We also need bounded distortion property (BDP),
which says that the value of the norm of derivatives cannot vary
too much. This is actually just a consequence of our previous assumptions.
And finally the boundary of $X$ is assumed to be ``smooth''
enough. For example all convex sets are like that. From this
property one may verify that the limit set is a Borel set. In the
finite case it is always compact. Together with OSC it also
guarantees some properties for the natural Borel regular
probability measure, so called conformal measure, associated to
this set.

This is one way to generalize similar kind of situation for finite
collection of similitudes; the setting introduced by Hutchinson in
\cite{hu}. Mattila proved in \cite{ma1} for the limit set $E$ of this
kind of setting the following result: Either $E$ lies on an
$l$--dimensional affine subspace or $\HH^t(E \cap M)=0$ for every
$l$--dimensional $C^1$--submanifold $M \subset \R^d$. Here $0<l<d$
and $\HH^t$ denotes the $t$--dimensional Hausdorff measure, where
$t=\dimh(E)$, the Hausdorff dimension of the set $E$. The main
result in this note is a generalization for this. Our approach is
based on an extensive use of tangents. Before going into more
detailed preliminaries we should mention that Springer has proved
in \cite{spr} similar result in the plane and Mauldin, Mayer and
Urba\'nski have studied similar behavior for connected limit sets
in \cite{mmu1} and \cite{mmu2}.

As usual, let $I$ be a countable set with at least two elements.
Put $I^* = \bigcup_{n=1}^\infty I^n$ and $I^\infty = I^\N = \{
(i_1,i_2,\ldots) : i_j \in I \text{ for } j \in \N \}$. Thus, if
$\iii \in I^*$ there is $k \in \N$ such that $\iii =
(i_1,\ldots,i_k)$, where $i_j \in I$ for all $j=1,\ldots,k$. We
call this $k$ as the \emph{length} of $\iii$ and we denote
$|\iii|=k$. If $\iii \in I^\infty$ we denote $|\iii|=\infty$. For
$\iii \in I^* \cup I^\infty$ we put $\iii|_k = (i_1,\ldots,i_k)$
whenever $1 \le k < |\iii|$.

Choose $\Omega$ to be open, bounded and connected set on $\R^d$.
Now for each $i \in I$ we define an injective mapping $\fii_i :
\Omega \to \Omega$ such that it is \emph{contractive}, that is,
there exists $0<s_i<1$ such that
\begin{equation} \label{eq:contactive}
  |\fii_i(x) - \fii_i(y)| \le s_i |x-y|
\end{equation}
whenever $x,y \in \Omega$. A mapping with equality in
(\ref{eq:contactive}) is called \emph{similitude}. We assume also
that $\fii_i$ is \emph{conformal}, that is, $|\fii_i'|^d =
|J\fii_i|$, where $J$ stands for normal Jacobian and the norm in the
left side is just the normal ``sup--norm'' for linear mappings. Here
the derivative exists $\HH^d$--almost all points using Rademacher's
theorem. This definition for conformality is usually better known
as $1$--quasiconformality; see V\"ais\"al\"a \cite{v}. Note that
a conformal mapping is always $C^\infty$ by theorem 4.1 of
Reshetnyak \cite{res}. We assume also that there exists a compact set $X$
with $\text{int}(X) \ne \emptyset$ such that $\fii_i(X) \subset X$
for every $i \in I$. The use of the bounded set $\Omega$ here is
essential, since conformal mappings contractive in whole $\R^d$
are similitudes, as $d$ exceeds 2. We call a collection $\{ \fii_i
: i \in I \}$ as \emph{conformal iterated function system} (CIFS)
if the following conditions (1)--(4) are satisfied:
\begin{itemize}
  \item[(1)] Mappings $\fii_i$ are \emph{uniformly contractive}, that is,
  $s:=\sup_{i \in I}s_i < 1$.
\end{itemize}
Denoting $\fii_\iii = \fii_{i_1} \pallo \cdots \pallo
\fii_{i_{|\iii|}}$ for each $\iii \in I^*$, we get from this
property that for every $n \in \N$
\begin{equation}
  d\bigl( \fii_{\iii|_n}(X) \bigr) \le s^n d(X)
\end{equation}
whenever $\iii \in I^\infty$. Here $d$ means the diameter of a
given set. Now we may define a mapping $\pi : I^\infty \to X$ such
that
\begin{equation}
  \bigl\{ \pi(\iii) \bigr\} = \bigcap_{n=1}^\infty
  \fii_{\iii|_n}(X).
\end{equation}
We set
\begin{equation}
  E = \pi(I^\infty) = \bigcup_{\iii \in I^\infty}
  \bigcap_{n=1}^\infty \fii_{\iii|_n}(X)
\end{equation}
and we call this set as the \emph{limit set} of the corresponding
CIFS. Our aim is to study this set. Observe that $E$ satisfies the
natural invariance equality, $E = \bigcup_{i \in I} \fii_i(E)$.
\begin{itemize}
  \item[(2)] \emph{Bounded distortion property} (BDP), that is, there
  exists $K \ge 1$ such that $|\fii_\iii'(x)| \le
  K|\fii_\iii'(y)|$ for every $\iii \in I^*$ and $x,y \in \Omega$.
\end{itemize}
For a finite collection of mappings, this is just a consequence of
smoothness (see lemma 2.2 of \cite{mu}) and in the infinite case it
follows from Koebe's distortion theorem as $d$ equals $2$ and
theorem 1.1 of \cite{u} whenever $d$ exceeds $2$.
Using these assumptions we may prove that each mapping $\fii_i$ is a
diffeomorphism and that there exists $D \ge 1$ such that
\begin{equation} \label{eq:bigD}
  D^{-1} \le \frac{d\bigl( \fii_\iii(E) \bigr)}{||\fii_\iii'||}
  \le D
\end{equation}
for every $\iii \in I^*$. Here $||\fii_\iii'|| = \sup_{x \in
\Omega} |\fii_\iii'(x)|$.
\begin{itemize}
  \item[(3)] \emph{Open set condition} (OSC) holds for $\text{int}(X)$,
  that is, $\fii_i\bigl( \text{int}(X) \bigr) \cap
  \fii_j\bigl( \text{int}(X) \bigr) = \emptyset$ for every $i \ne
  j$.
  \item[(4)] There exists $r_0 > 0$ such that
  \begin{equation}
    \inf_{x \in \partial X} \inf_{0<r<r_0} \frac{\HH^d\bigl( B(x,r) \cap
    \text{int}(X) \bigr)}{\HH^d\bigl( B(x,r) \bigr)} > 0,
  \end{equation}
  where $\partial X$ is the boundary of the set $X$.
\end{itemize}
We should mention that in \cite{mu}, instead of assumption (4), it
was used so called cone condition, which says that for each
boundary point $x$ there exists some ``cone'' in the interior of
$X$ with vertex $x$. However, assumption (4) is sufficient for our
setting, as was remarked also in \cite{mu}. Using these assumptions
it follows that $E$ is a Borel set. Suppose there exists a
Borel regular probability measure $m$ on $E$ such that
\begin{equation} \label{eq:confmeas}
  m\bigl( \fii_\iii(A) \bigr) = \int_A |\fii_\iii'(x)|^t dm(x),
\end{equation}
where $t=\dimh(E)$, $A \subset X$ is a Borel set and $\iii \in I^*$.
Then OSC and assumption (4) are crucial to derive that $m\bigl(
\fii_i(X) \cap \fii_j(X) \bigr) = 0$ for every $i \ne j$ (see
section 3 of \cite{mu} for details). If this measure exists, we call it a
\emph{$t$--conformal measure} and the corresponding CIFS
\emph{regular}. Observe that finite CIFS's are always regular (see
section 3 of \cite{mu} for details). If we consider measure theoretical
Jacobian $J_m$ for function $\fii_\iii$ defined as
\begin{equation}
  J_m\fii_\iii(x) = \lim_{r\searrow 0}
  \frac{m\bigl( \fii_\iii\bigl( B(x,r) \bigr) \bigr)}
       {m\bigl( B(x,r) \bigr)}
\end{equation}
for each point $x \in E$, we notice using conformality of
$\fii_\iii$ and (\ref{eq:confmeas}) that $J_m\fii_\iii^{-1}\bigl(
\fii_\iii(x) \bigr) = \bigl( J_m\fii_\iii(x) \bigr)^{-1} =
|\fii_\iii'(x)|^{-t} = \bigl|\bigl( \fii_\iii^{-1} \bigr)'\bigl(
\fii_\iii(x) \bigr)\bigr|^t$ for $m$--almost all $x \in E$ and for
all $\iii \in I^*$ (recall also theorem 2.12 of \cite{ma2}). Thus
for example, using BDP
\begin{equation} \label{eq:inverse}
  m\bigl( \fii_\iii^{-1}(A) \bigr) =
  \int_A |(\fii_\iii^{-1})'(x)|^t dm(x)
  \le ||(\fii_\iii^{-1})'||^t m(A)
  \le K^t ||\fii_\iii'||^{-t} m(A)
\end{equation}
whenever $A \subset \fii_\iii(X)$ is Borel. Furthermore by setting
$\Fii|_{\fii_i(X)}(x) = \fii_i^{-1}(x)$ for all $i \in I$ we get
at $m$--almost every point defined mapping $\Fii : \bigcup_{i \in
I} \fii_i(X) \to X$ for which
\begin{equation}
  m\bigl( \Fii(A) \bigr) = \int_A |\Fii'(x)|^t dm(x).
\end{equation}
In fact, $m$ is ergodic and equivalent to some invariant (with
respect to the function $\Fii$) measure (see section 3 of \cite{mu}
for details).
In the finite case, the $t$--conformal measure is \emph{Ahlfors
regular}, that is, there exists $C \ge 1$ such that
\begin{equation}
  C^{-1} \le \frac{m\bigl( B(x,r) \bigr)}{r^t} \le C
\end{equation}
for all $x \in E$ and $r>0$ small enough (see lemma 3.14 of
\cite{mu}).

Let $0<l<d$ be an integer and $G(d,l)$ be the collection of all
$l$--dimensional subspaces of $\R^d$. We denote by $P_V$ the
orthogonal projection onto $V \in G(d,l)$ and we put $Q_V = P_{V^\bot}$,
where $V^\bot$ is the orthogonal complement of $V$. For $x \in
\R^d$ we denote $V+x = \{ v+x : v \in V \}$. If $a \in \R^d$, $V
\in G(d,l)$ and $0<\delta<1$ we set
\begin{align}
  X(a,V,\delta) &= \{ x \in \R^d : |Q_V(x-a)| < \sqrt\delta|x-a| \} \\
  &= \{ x \in \R^d : d(x,V+a) < \sqrt\delta|x-a| \}. \notag
\end{align}
Note that the closure of $X(a,V,\delta)$ equals to the complement
of $X(a,V^\bot,1-\delta)$. We say that $V$ is a
\emph{$(t,l)$--tangent plane} for $E$ at $a$ if
\begin{equation}
  \lim_{r \searrow 0} \frac{m\bigl( B(a,r) \pois X(a,V,\delta) \bigr)}{r^t} = 0
\end{equation}
for all $0<\delta<1$. Furthermore we say that $V$ is a
\emph{strong $l$--tangent plane} for given set $A \subset \R^d$ at
$a$ if for every $0<\delta<1$ there exists $r > 0$ such that
\begin{equation}
  A \cap B(a,r) \subset X(a,V,\delta).
\end{equation}
For example, an $l$--dimensional $C^1$--submanifold has a strong
$l$--tangent plane at all of it's points. Observe that these two
tangents are exactly the same for $E$ in the case of Ahlfors
regular $m$. However, we shall not need this fact here. Recall
also that $t$--dimensional \emph{upper density} of a measure $\mu$
at $a$ is defined as
\begin{equation}
  \Theta^{*t}(\mu,a) = \limsup_{r \searrow 0}
  \frac{\mu\bigl( B(a,r) \bigr)}{r^t}.
\end{equation}

\section{Main result}

The main result of this note is the following theorem.

\begin{theorem}
  Suppose CIFS is given, $t=\dimh(E)$ and $0<l<d$. Then either
  $\HH^t(E \cap M) = 0$ for every $l$--dimensional $C^1$--submanifold
  $M \subset \R^d$ or the closure of $E$ is contained in some
  $l$--dimensional affine subspace or $l$--dimensional geometric
  sphere whenever $d$ exceeds $2$ and analytic curve if $d$ equals $2$.
\end{theorem}

Using this theorem we are able to find minimal amount of essential
directions for where the set $E$ is spread out. It also follows
that if $t$ is an integer, then the limit set is always either
$t$--rectifiable or purely $t$--unrectifiable. See Mattila
\cite{ma2} for definitions and properties for these concepts.
Provided that $d$ exceeds $2$, it is also easy to see that if at
least one of our conformal mappings is similitude, the
latter case of the theorem concerns only affine
subspaces.

The proof is divided into three parts. We start with an easy lemma
which provides a useful dichotomy for our purposes.

\begin{lemma}
  Suppose CIFS is given, $t=\dimh(E)$ and $0<l<d$. Then either
  $\HH^t(E \cap M) = 0$ for every $l$--dimensional $C^1$--submanifold
  $M \subset \R^d$ or the system is regular and at least one point of
  $E$ has a $(t,l)$--tangent plane.
\end{lemma}

\begin{proof}
  Assume $\HH^t(E \cap M) > 0$ for some $M$.  Since $\HH^t(E)<\infty$,
  the regularity of the system is
  guaranteed by theorem 4.17 of \cite{mu}. From 2.10.19(4) of
  Federer \cite{fe} we get that $\Theta^{*t}(m|_{E\pois M},x) = 0$ for
  $\HH^t$--almost every $x \in E \cap M$. Let $a \in E \cap M$ be
  such point and $V \in G(d,l)$ be a strong $l$--tangent plane for
  $M$ at $a$. This means that for any given $0<\delta<1$ there
  exists $r_\delta > 0$ such that
  \begin{equation}
    M \cap B(a,r) \subset X(a,V,\delta)
  \end{equation}
  whenever $r<r_\delta$. Thus for every $0<\delta<1$
  \begin{align}
    \limsup_{r \searrow 0} \frac{m\bigl( B(a,r)\pois X(a,V,\delta) \bigr)}{r^t}
    &\le \limsup_{r \searrow 0} \frac{m\bigl( B(a,r)\pois\bigl( M \cap B(a,r) \bigr) \bigr)}{r^t} \notag \\
    &= \Theta^{*t}(m|_{E\pois M},a) = 0
  \end{align}
  and $V$ is a $(t,l)$--tangent plane for $E$ at $a$.
\end{proof}

With this dichotomy in mind it is enough to study what happens if
one point of $E$ has a $(t,l)$--tangent plane.

\begin{theorem}
  Suppose CIFS is regular, $t=\dimh(E)$ and $0<l<d$. If one point
  of $E$ has a $(t,l)$--tangent plane then $m$--almost all of $E$
  is contained in the set $f(V)$, where $V \in G(d,l)$ and $f$ is some
  conformal mapping.
\end{theorem}

\begin{proof}
  Suppose $a \in
  E$ has a $(t,l)$--tangent plane $V \in G(d,l)$ and let $\eps > 0$.
  Now for each $j \in \N$ there exists $r_{j,0} > 0$ such that
  \begin{equation} \label{eq:measappr}
    m\bigl( B(a,r) \pois X(a,V,\tfrac{1}{j}) \bigr) \le \eps r^t
  \end{equation}
  whenever $r < r_{j,0}$. Let $\iii \in I^\infty$ be such that $\pi(\iii) =
  a$. For each $j \in \N$ choose some fixed radius $r_j < r_{j,0}$
  and $n_j \in \N$ such that
  \begin{align} \label{eq:incl}
    \fii_{\iii|_{n_j}}&(E) \subset B(a,r_j) \qquad \text{and}
    \notag \\
    \fii_{\iii|_{n_j}}&(E) \pois B(a,\tfrac{r_j}{2}) \ne
    \emptyset.
  \end{align}
  Since trivially $a \in \fii_{\iii|_n}(E)$ for all $n$, we
  have
  \begin{equation} \label{eq:appr1}
    \tfrac{r_j}{2} \le d\bigl( \fii_{\iii|_{n_j}}(E) \bigr) \le
    2r_j.
  \end{equation}
  For each $j$ define mapping $\psi_j : \R^d \to \R^d$ such that $\psi_j(x) =
  ||\fii_{\iii|_{n_j}}'||^{-1}(x-a)+a$. Then
  \begin{equation}
    \psi_j\bigl( X(a,V,\tfrac{1}{j}) \bigr) = X(a,V,\tfrac{1}{j})
  \end{equation}
  and
  \begin{equation}
    |\psi_j(x) - \psi_j(y)| = ||\fii_{\iii|_{n_j}}'||^{-1} |x-y|
  \end{equation}
  for every $x,y \in \R^d$. Now $F_j := \psi_j \pallo
  \fii_{\iii|_{n_j}}$ is clearly a conformal mapping from $\Omega$
  to $\R^d$. Since for arbitrary $x,y \in \Omega$
  \begin{align}
    K^{-1}|x-y| &\le ||\fii_{\iii|_{n_j}}'||^{-1}
    ||(\fii_{\iii|_{n_j}}^{-1})'||^{-1} |x-y| \notag \\
    &\le ||\fii_{\iii|_{n_j}}'||^{-1} |\fii_{\iii|_{n_j}}(x) - \fii_{\iii|_{n_j}}(y)|
    \\
    &= |F_j(x) - F_j(y)| \le |x-y| \notag
  \end{align}
  using BDP and mean value theorem, we notice that $F_j$
  is bi--Lipschitz with constants $K^{-1}$ and $1$ for every $j \in
  \N$. Using now Ascoli--Arzela's theorem we shall find an uniformly
  converging subsequence, say, $F_{j_k} \to F$, as $k \to \infty$.
  According to corollaries 37.3 and 13.3 of V\"ais\"al\"a \cite{v}
  we notice that $F^{-1}$ is conformal. Since
  \begin{align}
    m\bigl( E\pois F_j^{-1}\bigl( X(a,V,\tfrac{1}{j}) \bigr)
    \bigr) &= m\bigl( \fii_{\iii|_{n_j}}^{-1} \bigl( \fii_{\iii|_{n_j}}(E) \pois X(a,V,\tfrac{1}{j}) \bigr)
    \bigr) \notag \\
    &\le K^t ||\fii_{\iii|_{n_j}}'||^{-t} m\bigl( \fii_{\iii|_{n_j}}(E) \pois X(a,V,\tfrac{1}{j})
    \bigr) \notag \\
    &\le (DK)^t d\bigl( \fii_{\iii|_{n_j}}(E) \bigr)^{-t} m\bigl( B(a,r_j) \pois X(a,V,\tfrac{1}{j})
    \bigr) \\
    &\le (2DK)^t r_j^{-t} \eps r_j^t \notag
  \end{align}
  for every $j \in \N$ using (\ref{eq:inverse}), (\ref{eq:bigD}),
  (\ref{eq:incl}), (\ref{eq:measappr}) and (\ref{eq:appr1}), we conclude
  \begin{equation}
    m\bigl( E \pois F^{-1}(V+a) \bigr) \le (2DK)^t \eps.
  \end{equation}
  We finish the proof by letting $\eps \searrow 0$.
\end{proof}

Notice that the inclusion in the previous theorem
holds for the closure of $E$, since any set of full $m$--measure is
dense in $E$ and any $l$--dimensional $C^1$--submanifold is closed in
$\R^d$.
Now the main theorem follows as a corollary by just recalling that
conformal mappings are complex analytic in the plane and by
Liouville's theorem M\"obius transformations elsewhere (see theorem
4.1 of Reshetnyak \cite{res}).

\begin{remark}
  The proof of the main theorem was found in January '01 and it was
  supposed to be part of the author's thesis. Since recently
  there has been some interest for similar kind of questions (particularly
  \cite{mmu2}), it was decided to be published independently.
\end{remark}

\begin{ack}
  Author likes to thank professor Pertti Mattila for many useful
  discussions during the preparation of this note and the referee for
  pointing out couple of excellent remarks.
\end{ack}

\end{document}